\documentclass[11pt]{article} 

\usepackage[utf8]{inputenc} 
\usepackage{amsmath,amssymb,latexsym,amsthm}

\usepackage{geometry} 
\usepackage{graphicx} 
\usepackage{color,soul}
\usepackage{epstopdf}
\usepackage{booktabs} 
\usepackage{array} 
\usepackage{paralist} 
\usepackage{verbatim} 
\usepackage{subfigure} 

\usepackage{fancyhdr} 
\usepackage{sectsty}
\allsectionsfont{\sffamily\mdseries\upshape} 

\usepackage[nottoc,notlof,notlot]{tocbibind} 
\usepackage[titles,subfigure]{tocloft} 

\newtheorem{theorem}{Theorem}[section] 
\newtheorem{lemma}[theorem]{Lemma} 

\theoremstyle{definition} 
\newtheorem{definition}[theorem]{Definition}

\newtheorem{remark}[theorem]{Remark}

\newtheorem{example}{Example}

\newcommand{\R}{\mathbb{R}}
\newcommand{\Rn}{\R^{n}}

\newcommand{\N}{\mathbb{N}}

\newcommand{\abs}[1]{\left\vert#1\right\vert}
\newcommand{\set}[1]{\left\{#1\right\}}

\newcommand{\paren}[1]{\left ( #1\right )}

\newcommand{\wt}[1]{\widetilde{#1}}

\newcommand{\D}{\partial}

\def\XXint#1#2#3{{\setbox0=\hbox{$#1{#2#3}{\int}$ }
\vcenter{\hbox{$#2#3$ }}\kern-.6\wd0}}

\newcommand{\bs}{\backslash}
\newcommand{\p}{\partial}

\newcommand{\diam}[1]{\textrm{diam}\paren{#1}}
\newcommand{\supp}[1]{\textrm{supp}\paren{#1}}

\newcommand{\WF}[1]{\textrm{WF}\paren{#1}}

\newcommand{\WFp}[1]{\textrm{WF}^+\paren{#1}}
\newcommand{\WFm}[1]{\textrm{WF}^-\paren{#1}}
\newcommand{\WFpm}[1]{\textrm{WF}^\pm\paren{#1}}

\newcommand{\WFtm}[1]{\wt{\textrm{WF}^-}\paren{#1}}

\def\inter{\text{int}}

\renewcommand{\S}{\mathcal{S}}

\newcommand{\vareps}{\varepsilon}

\newcommand{\bna}{\begin{eqnarray*}}
\newcommand{\ena}{\end{eqnarray*}}
\newcommand{\bnan}{\begin{eqnarray}}
\newcommand{\enan}{\end{eqnarray}}

\newcommand{\E}{\mathcal E}

\title{Uniqueness for a Seismic Inverse Source Problem Modeling a Subsonic Rupture}
\author{Maarten V. de Hoop \thanks{Simons Chair in Computational and Applied Mathematics and Earth Science, Rice University, 6100 Main Street, Houston TX 77005, USA. \texttt{mdehoop@rice.edu}}, Lauri Oksanen \thanks{Department of Mathematics, University College London, Gower
Street, London UK, WC1E 6BT. \texttt{l.oksanen@ucl.ac.uk}} and Justin Tittelfitz \thanks{Department of Mathematics, Purdue University, West Lafayette IN 47907, USA. Currently at Amazon Web Services. \texttt{tttlf@amazon.com}}}

\date{} 

\begin{document}

\maketitle

\begin{abstract}
We consider an inverse problem for an inhomogeneous wave equation with discrete-in-time sources, modeling a seismic rupture. We assume that the sources occur along a path with subsonic velocity, and that data are collected over time on some detection surface. We explore the question of uniqueness for these problems, show how to recover the times and locations of sources microlocally, and then reconstruct the smooth part of the source assuming that it is the same at each source location.
\end{abstract}

\section{Introduction}

Let $c \in C^\infty(\R^n)$ be strictly positive    
and consider the wave equation 
\begin{align}
\label{PDEu}
	\begin{cases}
	\partial^2_{t} u - c(x)^2 \Delta u = F(t,x) & \textrm{ in } \R \times \Rn, \\
          u(0,\cdot) = \D_t u(0,\cdot) = 0 & \textrm{ in } \Rn.
	\end{cases}
\end{align}
We will study the inverse source problem to determine $F$
given the data
\begin{align*}
	\Lambda F := u\vert_{(0,T)\times\partial\Omega},
\end{align*}
where $\Omega \subset \Rn$ is an open and bounded set with smooth boundary.
It is well-known that such a problem does not have a unique solution in general. For example, if 
we set $F = \p ^2_{t} v - c^2 \Delta v$
where $v \in C_0^\infty(\Omega \times (0,T))$,
then $\Lambda F = 0$.

To overcome non-uniqueness we will assume that the source is of the form
\begin{align}
\label{def_F}
	F(t,x) := \sum_{j = 1}^J \delta(t - t_j) f_j(x), 
\end{align}
where $J \in \N$ and $0 < t_1 < t_2 < \ldots < t_J$.
Furthermore, 
we assume that $f_j$
is in the space of compactly supported distributions $\E'(\Omega)$, 
and has the form  
\begin{align}
\label{def_fj}
\langle f, \phi \rangle_{\E'\times C^\infty(\Omega)}
= \int_{S_j} h_j(x) \phi(x) dx, \quad \phi \in C^\infty(\Omega),
\end{align}
where $S_j = \supp{f_j}$ is a smooth oriented manifold with boundary and $h_j \in C^\infty(S_j)$.
We assume that either $\dim (S_j) = n$ or $\dim(S_j) = n-1$
and impose further conditions on the functions $f_j$ below.

The conditions are motivated by models of seismic ruptures,
where it is reasonable to assume that each $f_j$ has the same or similar spatial characteristics. 
It is also reasonable to model a rupture using discrete-in-time sources, as the sources only radiate where the velocity of the rupture changes, which only happens at discrete times \cite{Mad}.

We mention the widely applied procedure for estimating the source by Kikuchi \& Kanamori \cite{KK0}, which is based on maximizing the time correlations between observed and modeled wave solutions. Here, the ruptures are essentially represented by a sum of point sources parametrized by their locations and onset times. 
The sum of point sources models a sequence of subevents in the rupture.  
A refined, iterative procedure introduces in every iteration a new subevent \cite{LG}.
In our approach, 
we begin also by identifying the locations and onset times of subevents, however, in our case the subevents have spatial structure modeled by $f_j$.

\subsection{Previous literature}

Our proof uses the the unique continuation principle 
by Tataru \cite{T1},
see \cite{ROB} and  \cite{HORunique} for earlier results, and \cite{BL} and \cite{EINT} for extensions to other time-dependent systems like elasticity. In addition, we will draw upon ideas from the theory of inverse initial source problems, in particular, from \cite{SU} where a time-reversal approach for an inverse initial source problem with a non-constant wave speed was introduced. 
The motivation for \cite{SU} was the medical imaging modality known as thermoacoustic tomography but similar ideas have been used in many other applications, including geophysical ones; for time-reversal methods used in rupture detection, see \cite{ISHV}, \cite{WIS}, \cite{MLM}, \cite{RS}, \cite{EDMN}, \cite{KS}, \cite{KS2}, and in microseismicity see \cite{FJ}, \cite{DSFS}, \cite{APW}. 

The inverse initial source problem has been studied widely in the context of thermoacoustic and photoacoustic tomography: see \cite{SU}, \cite{XKA}, \cite{XWKA} for the problem with partial data, 
see \cite{SU2} for a speed with discontinuities, see \cite{QSUZ} for numerical discussion, see \cite{TIT}, \cite{TIT2} and \cite{Homan} for the problem in elastic and attenuating media respectively, and finally, one may find the surveys \cite{HKN}, \cite{KK1}, \cite{KK2} of interest. There has also been recent work on the problem of jointly recovering the speed and source \cite{SU3}, and the problem of recovery with an approximate speed \cite{OU}.

Let us now turn to inverse source problems where the source is on the right-hand side of the wave equation. We mention the result by two of the authors \cite{dHT},
where a source of the form (\ref{def_F}) is considered, but it is required that the sources are well-separated from one another in space and time, in contrast to the sub-sonic proximity required in the current work. These assumptions are appropriate for modelling microseismicity (instead of ruptures, as in this paper).
Most other results for inverse source problems consider a right-hand side of the form $a(t)f(x)$ or $a(t,x)f(x)$ where $a$ is a known function, see \cite{Y} and \cite{Puel} respectively,
and \cite{SU5} for a recent result. Similar problems have been stated and explored for the elastic wave equation, see \cite{GIY} and \cite{GY}. 

\section{Statement of the results}

We will reconstruct $F$ in two steps. First we use a microlocal argument to recover the source times $t_j$ and supports
$S_j = \supp{f_j}$, $j=1, 2, \dots, J.$
Then we impose an assumption that the distributions $f_j$
are translations of a single distribution $f$ and recover this. 

Before stating our results we need to introduce some notation.
We begin by recalling the definition of the wave front set, see e.g. 
\cite{Hor1} for further details.

\begin{definition}
Let $X \subset \R^n$ be open. 
The {\em wavefront set}  $\WF{w}$ of a distribution $w \in \mathcal{D}'(X)$ is a subset of the cotangent bundle $T^*X$ indicating the locations and the directions of the singularities of $w$. If $(x_0,\xi_0) \in T^* X \setminus 0$, then $(x_0,\xi_0)$ is not in the wavefront set of $w$ if there exists $\psi \in C_0^\infty(X)$ with $\psi(x_0) \neq 0$, and a conic neighborhood $V$ of $\xi_0$ such that
\begin{align*}
	\abs{\widehat{\psi w}(\xi)} \leq C_N (1 + \abs{\xi})^{-N},
\quad \xi \in V,\ N \in \N.
\end{align*}
Here $\widehat{\psi w}$ indicates the Fourier transform of $\psi w$.	
\end{definition}

If $w$ satisfies the wave equation
$$
\partial^2_{t} w - c(x)^2 \Delta w = 0, \quad \textrm{ in } \R \times \Rn, 
$$
then $\WF{w}$ is invariant under the bicharacteristic flow corresponding to the wave operator, see e.g. \cite{Hor3}. 
The principal symbol $p \in C^\infty(T^*\R^{1+n})$ of the wave operator is $p(t,x,\tau,\xi) = -\tau^2 + c^2(x)|\xi|^2$,
and the forward bicharacteristic flow $\Phi$ 
acts on the level set $p^{-1}(0) \subset T^*\R^{1+n}$ as follows
\begin{align*}
\Phi : \R \times p^{-1}(0) \to p^{-1}(0),
\quad
\Phi(s; t, x, \tau, \xi) = (t + s\tau, \gamma(s), \tau, \dot \gamma(s)),
\end{align*}
where $\gamma(s) = \gamma(s; x, \xi)$
is the geodesic on $(\Rn, c^{-2} dx^2)$ satisfying the initial conditions 
$\gamma(0) = x$ and $\dot \gamma(0) = \xi$. Here $\dot \gamma$ is the direction of $\gamma$ as a cotangent vector, that is, 
in coordinates $\dot \gamma = c^{-2} \sum_{j=1}^n (\p_s \gamma^j) dx^j$.

Let us now consider the solution $u$  of \eqref{PDEu}
where $F$ is of the form \eqref{def_F}.
For a set $A$, we denote by $\chi_A$ the indicator function of $A$,
that is, $\chi_A(x) = 1$ if $x \in A$ and $\chi_A(x) = 0$ otherwise. 
By Duhamel's principle, it holds that $u = \sum_{j=1}^J u_j$ where $u_j = \chi_{\set{t \geq t_j}}w_j$ and $w_j$ is the solution of 
\begin{align}
\label{PDEw}
	\begin{cases}
	\partial^2_{t} w - c^2 \Delta w = 0,  & \textrm{ in } \R \times \Omega, \\
	w(t_j,x) = 0; \quad \D_t w(t_j,x) = f_j(x)  & \textrm{ in } \Omega.
	\end{cases}
\end{align}
Note that $\WF{u_j}$ is not invariant under the bicharacteristic flow
$\Phi$ but $\WF{w_j}$ is. 

We will next formulate three assumptions in terms of microlocal properties of the distributions $f_j$ and $w_j$.
Firstly, we will assume that the wave front sets of $w_j$ are disjoint:
\begin{align}\tag{ML1}\label{ML1}
	\WF{w_j} \cap \WF{w_k} = \emptyset, \quad j \ne k.
\end{align}
Secondly, we define the $n-1$ dimensional manifold without boundary
\begin{align}
\label{def_Sigma}
\Sigma_j = \begin{cases}
\p S_j, & \dim(S_j) = n,
\\
S_j^\inter, & \dim(S_j) = n-1,
\end{cases}
\quad j = 1,\dots,J,
\end{align}
and assume that 
\begin{align}\tag{CN}\label{CN}
	\WF{f_j} = \overline{N^* \Sigma_j}, \quad j = 1, 2, \dots, J.
\end{align}
Here $N^* \Sigma_j$ 
is the conormal bundle of $\Sigma_j$.
In the case $\dim(S_j) = n-1$, we let $\nu$ to be one of the two unit conormal vector fields of $\Sigma_j$,
and in the case $\dim(S_j) = n$, we let $\nu$ to be the outward unit outward conormal vector field of $\Sigma_j$.
Then $\overline{N^* \Sigma_j}$ is the union of the following two sets  
$$
N_j^\pm = \{(x, a \nu) \in T^* \R^n;\ x \in \overline{\Sigma_j},\ \pm a > 0 \}.
$$ 
Note that if $\dim(S_j) = n-1$ then (\ref{CN}) amounts to assuming that the extension of $h_j$ in (\ref{def_fj}) by zero accross $\p S_j$ is smooth.
This follows from \cite[Th. 8.1.5]{Hor1}
together with a change of coordinates.
In the case $\dim(S_j) = n$, 
(\ref{CN}) means that the extension of $h_j$
as above is not smooth at any $x \in \p S_j$.

As $\WF{w_j}$ is invariant under the bicharacteristic flow,
it holds that $\WF{w_j(t, \cdot)}$ is the union of the two sets
$$
\WFpm{w_j(t, \cdot)} = 
\{(\gamma(s), \dot \gamma(s));\ s = t-t_j,\ \gamma = \gamma(\cdot; x, \xi),\ (x,\xi) \in N_j^\pm\}.
$$
We call $\WFp{w_j(t, \cdot)}$ and $\WFm{w_j(t, \cdot)}$ the outward and inward wavefronts, respectively. 
Note that the outward and inward wavefronts pair at the source time $t_j$ 
in the following sense
\begin{align}\label{pair}
	\WFp{w_j(t_j,\cdot)} = N_j^+ = \widetilde {N_j^-} = \WFtm{w_j(t_j,\cdot)},
\end{align}
where tilde indicates reflection in the dual variable, that is,
\begin{align}
\label{def_dual_reflection}
	\wt A = \set{(x,-\xi) : (x,\xi) \in A}, \quad A \subset T^*\Omega.
\end{align}
Now we state our third microlocal assumption, that \eqref{pair} is the only kind of pairing. That is, we assume that the manifolds $\Sigma_j$ are connected and 
\begin{align}\label{ML2}\tag{ML2}
\text{if ${\textrm{WF}^\sigma\paren{w_j(t,\cdot)}} = \wt{\textrm{WF}^{\sigma'}}\paren{w_k(t,\cdot)}$ then $j=k$, $t=t_j$ and $\sigma \ne \sigma'$.}
\end{align}

We denote by $S^* \Omega$ the unit cosphere bundle 
$$
S^* \Omega = \{(x,\xi) \in T^* \Omega;\ c^2(x)|\xi|^2 = 1\}.
$$
We will prove the following theorem in Section \ref{identification}.

\begin{theorem}\label{th_microlocal} 
Suppose that $\Omega$ is non-trapping and strictly convex in the sense that 
for all $(x,\xi) \in S^*\Omega$ 
the geodesic $\gamma = \gamma(\cdot; x, \xi)$ 
satisfies the following: 
there is unique $s \in (0,T-t_J)$ such that $\gamma(s) \in \p \Omega$,
and, furthermore, 
$\dot \gamma(s) \notin T_{\gamma(s)}^* \p \Omega$
and $\gamma(t) \in \Rn \setminus \overline \Omega$ for all $t > s$.
Suppose that the manifolds $\Sigma_j$ are smooth and connected, and that \eqref{CN}, \eqref{ML1} and \eqref{ML2} are satisfied. 
Then the times $t_j$ and supports 
$S_j$, $j=1,2,\dots, J$, can be recovered from the boundary data $\Lambda F$.
\end{theorem} 

\begin{figure}
\centering
  \includegraphics[width=7cm]{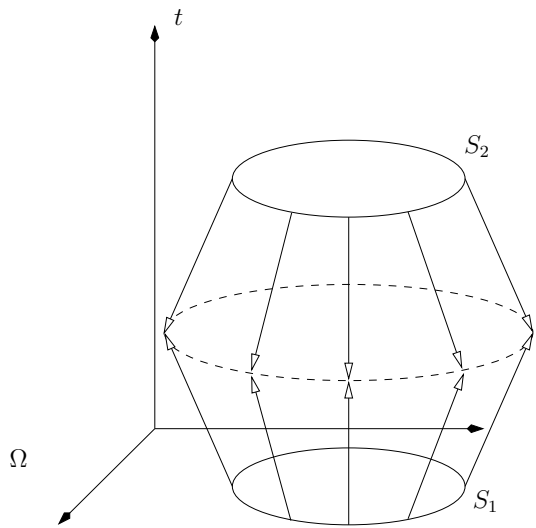}
\caption{Two wavefronts forming a spurious pairing.}
\label{pairing}
\end{figure}

Let us give an example that does not satisfy \eqref{ML2}. Let $n=2$, and let $S_1$ and $S_2$ be two identical discs, so that the initial singular supports (that is, the projections of $\WF{f_j}$, $j=1,2$, to the base space $\Omega$) are circles. Suppose that the wave speed $c$ is constant. Then at $\frac12 (t_2 - t_1)$ the outgoing wavefront from the first source and the inward wavefront from the second source pair to form a larger circle, see Figure \ref{pairing}.
Note however that if the spatial location of either of these discs is perturbed slightly, this pairing no longer occurs. Thus this type of spurious pairing does not happen generically, and so we view \eqref{ML2} as a mild hypothesis. 

We will next consider the case that the distributions $f_j$
are obtained from a single distribution $f$ via translations, and show how to recover $f$.
We will study two translation models: a Riemannian parallel transport and the Euclidean translation.  
To simplify the notation, we assume that $\Omega$ contains the origin of $\Rn$.

Let us consider the Riemannian case first. 
We assume that the Riemannian manifold $(\Omega, c^{-2} dx^2)$
is simple, that is, $\Omega$ is simply connected, $\p \Omega$ is strictly convex in the sense of the second fundamental form, and there are no conjugate points on $\Omega$.
We denote by 
$$
\mathcal T_x : T_0 \Omega \to T_x \Omega, \quad x \in \Omega,
$$
the parallel transport along the radial unit speed geodesic from the origin to $x$.
That is, for each $x$ we choose $\xi \in S_0^* \Omega$ 
and $r \ge 0$ such that $x = \gamma(r;0,\xi)$, and for each vector $v \in T_0 \Omega$ we solve the equation
$$
D_s V = 0, \quad V(0) = v,
$$
where $D_s$ is the covariant derivative of the metric $c^{-2} dx^2$
along the curve $\gamma(s;0,\xi)$.
Finally we set $\mathcal T_x v = V(r)$.

We assume that for each $j=1, 2, \dots, J$, there is $x_j \in \Omega$ such that
\begin{align}\tag{R1}\label{R1}
	f_j = f \circ \mathcal T_{x_j}^{-1} \circ \exp_{x_j}^{-1},
\end{align}
where $\exp$ is the exponential map of $(\Omega, c^{-2} dx^2)$,
$f \in \E'(T_0 \Omega)$, and the precomposition means the pullback
of $f$ by $\mathcal T_{x_j}^{-1} \circ \exp_{x_j}^{-1}$,
see e.g. \cite[Th. 6.1.2]{Hor1} for the definition. 
Note that if $c=1$ identically, then in coordinates, $\mathcal T_x$ is the identity and 
$$
f(v) = f_j \circ \exp_{x_j}(v) = f_j(x_j + v).
$$
Thus $f_j$ is obtained from $f$ by an Euclidean translation.

We assume that 
\begin{align}\tag{SS}\label{SS}
	d(x_{j+1},x_j) < t_{j+1} - t_{j}, \quad j = 1,2,\dots, J-1,
\end{align}
where $d(\cdot, \cdot)$ is the Riemannian distance function of $(\Rn, c^{-2} dx^2)$.
Note that $d(x,y)$ gives the travel time distance between points $x,y \in \Rn$.
We think of (\ref{SS}) as a condition limiting the speed at which the sources can propagate, effectively requiring this motion to be ``sub-sonic'', i.e. slower than the speed of wave propagation.
Let us emphasize
that the translation model (\ref{R1}) considers only spatial variables and
says nothing about the speed of the translation in spacetime whereas
(\ref{SS}) requires that the speed is sub-sonic.

We will prove the following theorem in Section \ref{sec_recovery_smooth}.

\begin{theorem}\label{uniqueRie}
Suppose that the Riemannian manifold $(\Omega, c^{-2} dx^2)$
is simple and that \eqref{SS} and \eqref{R1} are satisfied. Suppose furthermore that 
the times $t_j$ and the points $x_j$, $j=1,2,\dots,J$, are known. 
If 
\begin{align}
\label{T_cond}
T > t_1 + \diam{\Omega},
\end{align}
where $\diam{\Omega} = \sup_{x,y \in \Omega} d(x, y)$,
then $F$ can be recovered from the boundary data $\Lambda F$.
\end{theorem}

In order to combine Theorems \ref{th_microlocal} and \ref{uniqueRie}
we need to be able to determine the points $x_j$ given the supports $S_j$.
We will consider this problem in Section \ref{sec_center}.
 
Let us now describe the Euclidean translation model. 
We assume that for each $j=1, 2, \dots, J$, there is $x_j \in \Omega$ such that
\begin{align}\tag{E1}\label{E1}
	f_j(x) = f(x - x_j),
\end{align}
where $f \in \E'(\Omega)$.
Furthermore we assume that in addition to the sub-sonic condition (\ref{SS}) the following separation condition holds:
\begin{align}\tag{E2}\label{E2}
	t_2 - t_1 > \frac{1 - c^-/c^+}{1 - \rho} R
\end{align}
where $c^+ = \sup_{x \in \Omega} c(x)$, $c^- = \inf_{x \in \Omega} c(x)$
and 
\begin{align}
\label{def_rho}
\rho = \max_{j=1,\dots,J-1} \frac{d(x_{j+1},x_j)}{t_{j+1} - t_{j}},
\quad R = \max_{j=1,\dots,J}\min \{ r > 0;\ S_j \subset B_r(x_j)\}.
\end{align}
Here $B_r(x)$ is the closed geodesic ball $\{y \in \Rn;\ d(y,x) \le r \}$.
Note that (\ref{SS}) implies that $\rho \in [0,1)$.
We will prove the following theorem in Section \ref{sec_recovery_smooth}. 

\begin{theorem}\label{uniqueEuc}
Suppose that the Riemannian manifold $(\Omega, c^{-2} dx^2)$
is simple and that \eqref{SS}, \eqref{E1} and \eqref{E2} are satisfied. Suppose furthermore that 
the times $t_j$ and the points $x_j$, $j=1,2,\dots,J$, are known. 
If $T$ satisfies (\ref{T_cond}), then $F$ can be recovered from the boundary data $\Lambda F$.
\end{theorem}

If $c$ is constant, then \eqref{R1} and \eqref{E1} are equivalent and 
\eqref{E2} is trivially satisfied.
Without loss of generality we may assume that $f$ is defined so that 
the center of mass of its support is at the origin. 
Then $x_j$ is the center of mass of $S_j$
and therefore $S_j$ determines $x_j$, see Section \ref{sec_center} for more details. 
We will give further examples in Section \ref{examples}.

Let us formulate one more result where, instead of a translation assumption as above, we assume the following strong separation condition: 
\begin{align}\tag{TS}\label{TS}
	(1-\rho)(t_j - t_{j-1}) > 2R,
\end{align}
where $\rho$ is as in (\ref{def_rho}).
This condition not only limits the speed at which the source can move, but it also implies a minimum gap in time between sources (of size roughly $2R$). This condition is stronger than \eqref{E2}, but has the advantage of allowing completely distinct $f_j$
and arbitrary geometry $(\Omega, c^{-2}dx^2)$.
We prove the following theorem in Section \ref{sec_recovery_smooth}.

\begin{theorem}\label{uniqueVS}
Suppose that the conditions \eqref{SS} and \eqref{TS} 
are satisfied, and that the times $t_j$ and the supports $S_j$, $j=1,\dots,J$, are known. If $T > t_J + \diam{\Omega}$ then $F$ can be recovered from the boundary data $\Lambda F$.
\end{theorem}

\section{Microlocal identification}\label{identification}

In this section we prove Theorem \ref{th_microlocal}.
We define the exit time 
$$
\sigma_\Omega(x,\xi) = \max\set{t \geq 0 : \gamma(t; x,\xi) \in \overline\Omega}, \quad (x,\xi) \in T^* \Rn \setminus 0,\ x \in \overline \Omega.
$$
Let $t \in \R$ and consider the map 
$$
\Psi_t : T^* \Omega \setminus 0 \to T^* (\R \times \p \Omega),
\quad \Psi_t(x, \xi) = (t + \sigma \tau, \gamma(\sigma), \tau, \dot \gamma'(\sigma)),
$$
where $\tau = c(x)|\xi|$, $\sigma = \sigma_\Omega(x,\xi)$, $\gamma = \gamma(\cdot; x, \xi)$
and $\dot \gamma'$ is the projection of $\dot \gamma$ on $T^* \p \Omega$.
Note that $\Psi_t$ is the composition of the restriction on $\{t\} \times \Omega$, the bicharacteristic flow $\Phi$, and the restriction on $T^* (\R \times \p \Omega)$.
It is well-known that $\Psi_t$
is a local diffeomorphism if $\Omega$ is non-trapping and strictly convex.
For the convenience of the reader we give a proof here. 

\begin{lemma}
\label{lem_Psi}
Suppose that $\Omega$ is non-trapping and strictly convex as
formulated in Theorem \ref{th_microlocal}. Let $t \in \R$. Then $\Psi_t$ is an injective local diffeomorphism.
\end{lemma}
\begin{proof}
We begin by showing that $\sigma_\Omega$ is smooth on $S^* \Omega$.
Let $(x_0, \xi^0) \in T^* \Omega \setminus 0$. By the non-trapping assumption 
$s_0 := \sigma_\Omega(x_0,\xi^0)$ is well-defined and by the convexity assumption $\dot \gamma(s_0; x_0, \xi^0)$ is not tangential to $\p \Omega$.
It follows from the implicit function theorem that the equation 
$\gamma(s;x,\xi) \in \p \Omega$ has a unique solution $s$
near $s_0$ that depends smoothly on $(x,\xi)$ near $(x_0, \xi^0)$.
By the convexity assumption this solution coincides with $\sigma_\Omega$
near $(x_0, \xi^0)$. This shows that $\sigma_\Omega$ is smooth
and therefore $\Psi_t$ is smooth.

We will use boundary normal coordinates 
$y := (y^1, y') \in (-\epsilon, \epsilon) \times \p \Omega$ where $\epsilon > 0$ is small.
In these coordinates the metric tensor $g := c^{-2} dx^2$ has the form 
$$
g(y) = \begin{pmatrix}
1 & 0 \\
0 & h(y)
\end{pmatrix}.
$$
We denote by $|\eta|_g$ the norm of a cotangent vector 
$\eta = (\eta_1, \eta')$ 
with respect to the metric $g$, and have
$|\eta|_g^2 = \eta_1^2 + |\eta'|_h^2$.

We show next that $\Psi_t$ is an immersion.
Let $(x_0, \xi^0) \in T^* \Omega \setminus 0$ and define $s_0$ as above. 
We denote $\phi(x,\xi) = \gamma(s_0; x, \xi)$
and $\psi(x,\xi) = \dot \gamma(s_0; x, \xi)$.
Let
$p \in T_{(x_0, \xi^0)} T^* \Omega$ satisfy $d\Psi_t p = 0$.
The third component of this equation says that $d\tau p = 0$ and 
therefore the first component implies that $d\sigma p = 0$.
Now the second and fourth components imply $d\phi p = 0$ and $d\psi' p = 0$.
As the geodesic flow is a diffeomorphism on $T^* \Rn$,
$d\phi p = 0$ and $d\psi p = 0$ imply that $p=0$.
Thus it is enough to show that $d\psi_1 p = 0$.
As the geodesic flow preserves the norm, we have
$$
0 = d\tau p = d|\psi|_g^2 p = 2 \psi_1 d\psi_1 p + d |\psi'|_h^2 p.
$$
As $\dot \gamma(s_0; x_0, \xi^0)$ is not tangential to $\p \Omega$,
we have $\psi_1 \ne 0$.
Moreover, 
$$
d |\psi'|_h^2 p
= 2 \psi_j h^{jk} d \psi_k p + \psi_j dh^{jk} d\phi p \psi_k  = 0,
$$
whence $d\psi_1 p = 0$ and we have shown that $\Psi_t$ is an immersion.
As $T^* \Omega \setminus 0$ and $T^* (\R \times \p \Omega)$ have the same dimension, $\Psi_t$ is a local diffeomorphism. 

It remains to show that $\Psi_t$ is injective. Suppose that 
$(r, y, \tau, \eta') \in T^* (\R \times \p \Omega)$
and that there is $(x,\xi) \in T^* \Omega \setminus 0$
such that $\Psi_t(x,\xi) = (r, y, \tau, \eta')$.
Then $|\eta'|_g \le \tau$ and there is a unique $a \ge 0$
such that $|\eta' + a \nu|_g = \tau$
where $\nu$ is the outward unit normal covector of $\p \Omega$.
By the convexity assumption $\gamma$ 
does not return to $\overline \Omega$ after $\sigma$,
whence $\dot \gamma(\sigma) = \eta' + a \nu$.
We have $\sigma = (r - t) / \tau$ and 
$(x,\xi) = (\beta(\sigma), \dot \beta(\sigma))$ 
where $\beta = \gamma(\cdot; y, -\eta' - a \nu)$.
\end{proof}

\begin{proof}[Proof of Theorem \ref{th_microlocal}]
We recall that 
$$
\Lambda F = u|_{(0,T) \times \p \Omega} 
= \sum_{j=1}^J \chi_{\set{t \geq t_j}}w_j|_{(0,T) \times \p \Omega}.
$$
The map $f_j \mapsto w_j|_{(t_j, T) \times \p \Omega}$
is a sum of two elliptic Fourier integral operators 
with canonical relations given by the graphs of $\Psi_{t_j}$
and the composition of 
the reflection (\ref{def_dual_reflection}) and 
$\Psi_{t_j}$ respectively, see e.g. \cite[Prop. 3]{SU}.
As $\WF{f_j}$ is symmetric with respect to the reflection (\ref{def_dual_reflection}), we consider only $\Psi_{t_j}$.
The assumption that unit speed geodesics exit $\Omega$ before time $T - t_J$
together with (\ref{ML1}) implies that 
$$
\WF{\Lambda F} = \bigcup_{j=1}^J \Psi_{t_j}(\WF{f_j}).
$$

By Lemma \ref{lem_Psi}, the map $\Psi_{t}$ is 
continuous  
and therefore it maps the connected components $\WFpm{w_j(t,\cdot)}$ of $\WF{w_j(t,\cdot)}$
to connected components of $\Psi_{t}(\WF{w_j})$
assuming that $\WF{w_j(t,\cdot)} \subset T^* \Omega$.
Let us consider two connected components $\Gamma_1$
and $\Gamma_2$ of $\WF{\Lambda F}$ and let $t \in (t_0, t_1)$ 
where $t_0 \in \R$ is chosen to be the smallest possible time 
so that $\Psi_t^{-1}(\Gamma_1 \cup \Gamma_2)$ is well-defined (that is, the image stays in $T^* \Omega$) and 
\begin{align*}
t_1 = \min \{r \in \R;\ &\text{there are $(y, \eta') \in T^* \p \Omega$
and $\tau \in \R$ such that}  
\\&
(r,y,\tau,\eta') \in \Gamma_1 \cup \Gamma_2\}.
\end{align*}
Then $\Psi_t^{-1}(\Gamma_p) = \textrm{WF}^{\sigma_p}\paren{w_{j_p}(t,\cdot)}$, $p=1,2$,
for some $\sigma_p = \pm$ and $j_p = 1, \dots,J$.
By (\ref{ML2}) the sets $\Psi_t^{-1}(\Gamma_p)$, $p=1,2$, pair under the reflection (\ref{def_dual_reflection}) if and only if 
$j_1 = j_2$, $t = t_{j_1}$ and they coincide with the sets $N_{j_1}^\pm$.

The assumption (\ref{ML1}) implies that 
there is a bijection between the connected components of $\WF{\Lambda F}$
and the sets $N_j^\pm$, $j=1,\dots,J$.
Thus we can determine the times $t_j$ and the sets $N_j^\pm$, $j=1,\dots,J$.
\end{proof}

We get the following partial data result 
by inspecting the proof of Theorem \ref{th_microlocal}:

\begin{remark}
\label{rem_partial_data}
Consider the case where we know only a restriction of $\Lambda F$, 
that is, we know $u|_{(0,T) \times \omega}$ where $\omega \subset \p \Omega$
is open. Then we can still recover the source times $t_j$, $j=1,2,\dots,J$,
assuming a stronger form of (ML2). That is, 
the connected components $\Gamma_k$, $k=1,2,\dots,K$,
of $\WF{u|_{(0,T) \times \omega}}$ are assumed to
form pairs exactly at times $t_j$ in the sense that if 
\begin{align}
\label{partial_pairing}
\Psi_t^{-1}(\Gamma_{k_1}) \cap \widetilde{\Psi_t^{-1}(\Gamma_{k_2})} \ne \emptyset
\end{align}
then $t = t_j$ for some $j$ and that for all $j$ there are $k_1$ and $k_2$
such that (\ref{partial_pairing}) holds with $t = t_j$.
\end{remark}

The condition in Remark \ref{rem_partial_data}
means firstly that $\omega$ needs to be large enough so that we catch parts of all outward and inward 
wavefronts and that the outward and inward 
parts coming from the same source do not miss each other completely when propagated back using 
$\Psi_t^{-1}$, and secondly, that there are no spurious pairings.

Note that if $\Psi_t^{-1}(\Gamma_{k_1}) \subset \WFp{w_j(t, \cdot)}$
and $\Psi_t^{-1}(\Gamma_{k_2}) \subset \WFm{w_j(t, \cdot)}$
then the projection of the intersection (\ref{partial_pairing})
on the base space $\Omega$
is a subset of $\Sigma_j$ assuming that there are no spurious pairings. We can reconstruct this subset, but 
typically we can not reconstruct the whole set $\Sigma_j$ from the partial data
by using the above microlocal argument. We will further discuss the partial data case in Remark \ref{rem_partial_data2} below.

\begin{remark}
In a procedure introduced by Ishii \textit{et al.} \cite{Ishii_2005}
and quite commonly applied in seismology, the wavefield observed in
(an open subset of) the boundary is reverse-time continued and then
restricted to a subset of a chosen hypersurface, $\Sigma \subset
\Omega$ say, 
yielding $\sum_{j=1}^J w_j |_{\Sigma}$ without determining the $t_i$
explicitly. As a matter of fact, this is done microlocally and
referred to as backprojection with stacking (over the point receivers
in the mentioned subset of the boundary). In the case $\dim S_i = n$,
we can extend this procedure using our model as follows: If $S_i \cap
\Sigma \ne \emptyset$ and there are no spurious pairings, then the
paired components of the wavefront set of $\sum_{j=1}^J w_j
|_{\Sigma,t=t_i}$ correspond to the two components of the conormal
bundle of $S_i \cap \Sigma$ in $T^* \Sigma$, and this pairing can be
recovered by our method.
\end{remark}

\section{Reconstruction of the smooth part of the source}

\subsection{Distances to geodesic balls}

We begin by establishing two lemmas. Here $(M,g)$ is a smooth compact Riemannian manifold with boundary. We define 
$$
	\sigma_p(\xi) = \sup \{t > 0;\ \exp_p(t\xi) \in M^\inter\},
	\quad p \in M^\inter,\ \xi \in S_p M,
$$
where $S M$ denotes the unit sphere bundle of $M$,
and $B_r(p) = \{x \in M;\ d(x,p) \le r\}$, $r > 0$,
where $d$ denotes the distance function of $M$.

\begin{lemma}
	\label{lem_xy_radial}
	Suppose that $\p M$ is strictly convex in the sense of the second fundamental form. Let $p \in M^\inter$ and let $R > 0$.  Suppose that $S:= B_R(p) \subset M^\inter$, and that $\p S$ is smooth. Let $y \in \p M$ and 
suppose that $x \in S$ satisfies 
	\begin{align}
		\label{x_closest}
		d(y,x) = d(y,S).
	\end{align}
	Then there is $\xi \in S_p M$ such that 
	\begin{align}
		\label{xy_radial}
		x = \exp_p(R\xi) \quad \text{and} \quad y = \exp_p(\sigma_p(\xi) \xi).
	\end{align}
\end{lemma}

\begin{proof}
	Clearly $x \in \p S$ and there is $\xi \in S_p M$ such that $x = \exp_p(R\xi)$. Let $\gamma : [0, \ell] \to M$ be a shortest path from $x$ to $y$. Then $\gamma$ is $C^1$ and we may assume without loss of generality that it has unit speed \cite{Alexander1981}.
	
	A shortcut argument shows that $\dot \gamma(0) \perp \p S$. Thus $\gamma(t)$ coincides with the path $\wt\gamma(t) = \exp_p((t+R) \xi)$ until it hits the boundary $\p M$ at $t = \sigma_p(\xi)-R$ (see Figure \ref{normal_fig}). As $\p M$ is strictly convex $\dot \gamma(\sigma_p(\xi)-R)$ is not tangential to the boundary $\p M$. This implies that $\sigma_p(\xi)-R = \ell$, since otherwise $\gamma$ can not be $C^1$.
\end{proof}

\begin{figure}
\centering
  \includegraphics[width=7cm]{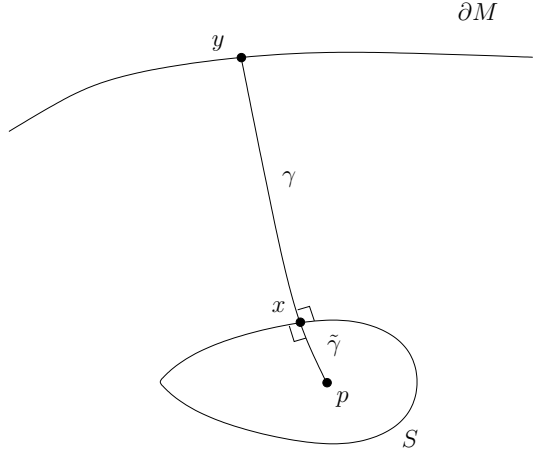}
\caption{The geodesics $\gamma$ and $\wt\gamma$ coincide.}
\label{normal_fig}
\end{figure}

In general there might exist $x \in \p S$ such that 
$$
	d(y, x) > d(y,S), \quad \text{for all $y \in \p M$}.
$$
However, in the case of a simple manifold this can not happen. 

\begin{lemma}
\label{lem_simple_radial}
	Suppose that $(M,g)$ is simple. Let $p \in M^\inter$ and let $R > 0$. Suppose that $S := B_R(p) \subset M^\inter$. Let $\xi \in S_p M$ and define $x \in \p S$ and $y \in \p M$ by \eqref{xy_radial}.
	Then \eqref{x_closest} holds.
\end{lemma}
\begin{proof}
	Note that $\p B_R(p)$ is smooth.
	As $S$ is compact, there is a point $z \in \p S$ such that $d(y,z) = d(y,S)$.
	Lemma \ref{lem_xy_radial} implies that there is $\zeta \in S_p M$ such that
	\begin{align*}
		z = \exp_p(R\zeta) \quad \text{and} \quad y = \exp_p(\sigma_p(\zeta) \zeta).
	\end{align*}
	The map $\exp_p$ is injective by the simplicity, whence $\zeta = \xi$.
	In particular, $z = x$ and \eqref{x_closest} holds.
\end{proof}

\subsection{Unique continuation}

We recall that the following time-sharp semi-global unique continuation result follows from the seminal local result by Tataru \cite{T1}.

\begin{theorem}
\label{th_uniq_cont}
Let $h \in C(\p \Omega)$ 
and define 
$$
\Gamma(h) = \{(t,y) \in \R \times \p \Omega;\ |t| < h(y) \},
\quad T = \max_{y \in \p \Omega} h(y).
$$
Let $s \in \R$, and suppose that 
$w \in H^{s}((-T,T) \times \R^n)$
satisfies $\p_t^2 w -c^2 \Delta w = 0$ and
\begin{align}
\label{vanishing_caucy_data}
w|_{\Gamma(h)} = 0, \quad \p_\nu w|_{\Gamma(h)} =0.
\end{align}
Then $w = 0$ and $\p_t w = 0$ on $\{0\} \times \Omega(h)^\inter$,
where
$$
\Omega(h) = \{x \in \Omega;\ \text{there is $y \in \partial \Omega$ such that $d(x,y) \leq h(y)$}\}.
$$
\end{theorem} 

See \cite[Th. 3.16]{Katchalov2001} for a proof in the case $s=1$ and $\Gamma(h)$ is a cylinder in $(-T,T) \times \p \Omega$. The general case, that is folklore, can be reduced to this special case by approximating $\Gamma(h)$ with a union of cylinders and by approximating $w$ with a smooth function. 
We refer to \cite{SU} where an analogous reduction using cylinders is given, and discuss only the smooth approximation of $w$.
Note that the traces of $w$ in Theorem \ref{th_uniq_cont}
are well-defined in the sense of \cite[Corollary 8.2.7]{Hor1} since $\WF{w}$ is a subset of the characteristic set $p^{-1}(0)$. 

Let $\epsilon > 0$, $\psi \in C_0^\infty(-\epsilon,\epsilon)$, let us extend $w$ 
by zero to $\R \times \R^n$
while denoting the extension still by $w$,
and let $\tilde w$ be the convolution in the time variable $\tilde w = \psi * w$.
As the operator $\p_t^2 - c^2 \Delta$
commutes with the map $w \mapsto \psi*w$,
the distribution $\tilde w$ satisfies
\begin{align}
\label{wave_eq_tildew}
\p_t^2 \tilde w -c^2 \Delta \tilde w = 0 \quad 
\text{in $I_\epsilon \times \R^n$},
\end{align}
where $I_\epsilon = (-T+2\epsilon, T-2\epsilon)$.
Moreover, (\ref{vanishing_caucy_data}) 
implies that 
$\tilde w = 0$ and $\p_\nu \tilde w =0$
on $\Gamma(h-2\epsilon)$.
We will show below that $\tilde w \in C^\infty(I_\epsilon \times \R^n)$, and therefore we may apply Theorem \ref{th_uniq_cont}
with $s=1$ to obtain $\tilde w = 0$
and $\p_t \tilde w = 0$ on $\{0\} \times \Omega(h-2\epsilon)$.
Letting $\psi \to \delta$
in the sense of distributions and $\epsilon \to 0$, we conclude that 
$w = 0$ and $\p_t w = 0$ on $\{0\} \times \Omega(h)^\inter$.
It remains to show that $\tilde w$ is smooth.
Clearly $\tilde w \in C^\infty(I_\epsilon; H^s(\R^n))$
and (\ref{wave_eq_tildew}) implies that 
$\Delta \tilde w(t) \in H^s(\R^n)$, $t \in \R$.
Thus $\tilde w(t) \in H^{s+2}(\R^n)$, $t \in \R$,
and we see that $\tilde w$ is smooth by using an induction.

\subsection{Recovery under the translation and separation conditions}
\label{sec_recovery_smooth}

To simplify the notation, we will assume below without loss of generality that $t_1 = 0$ and $x_1 = 0$. 

\begin{lemma}
\label{lem_SS}
Let $x_j \in \Omega$, $j=1,2,\dots,J$ satisfy (\ref{SS}) 
and define $\rho \in [0,1)$ by (\ref{def_rho}).
Let $r > 0$. Then for any $j,k=1,\dots,J$
and any $y \in B_r(x_j)$ there exists $x \in B_r(x_k)$ so that
\begin{align*}
	d(x,y) \leq \rho \abs{t_j - t_{k}}.
\end{align*}
\end{lemma}
\begin{proof}
Suppose first that $j < k$ and note that (\ref{SS}) implies that 
$$
d(x_k, x_j) \le d(x_k, x_{k-1}) + d(x_{k-1}, x_{k-2}) +\dots + d(x_{j+1}, x_j)
= \rho (t_{k} - t_{j}).
$$
Combining this with an analogous computation in the case $j > k$ yields
\begin{align}
\label{SS2}
d(x_k, x_j) \le \rho|t_{k} - t_{j}|, \quad j,k = 1,\dots,J. 
\end{align}
Let $x$ be the closest point in $B_r(x_k)$ to $y$. Then the geodesic from $y$ to $x$ hits $\partial B_r(x_k)$ normally by \cite[Corollary 26]{On}, whence $d(y,x) = d(y,x_k) - d(x_k,x) = d(y,x_k) - R$. 
We conclude by observing that $d(y,x_k) \leq d(y,x_j) + d(x_j,x_k) \le R + \rho\abs{t_j - t_k}$. 
\end{proof}

The recovery of the smooth part is based on finite speed of propagation and unique continuation as described in the following two lemmas respectively. Briefly, first we will show that there is a gap in time where only signals from the first source have arrived; this is illustrated in Figure \ref{fspucp_fig}. Then we use unique continuation to determine $f_1$ in part of $S_1$.

\begin{lemma}\label{fsp}
Let $x_j \in \Omega$, $j=1,2,\dots,J$ satisfy (\ref{SS}) 
and define $\rho \in [0,1)$ and $R > 0$ by (\ref{def_rho}).
Consider the solutions $w_j$, $j=1,2,\dots,J$, of (\ref{PDEw}).
	If $(t,y) \in (0,T) \times \partial\Omega$ satisfies
	\begin{align*}
          		t \le d(y,B_R(x_1)) + (1-\rho)t_2,
	\end{align*}
	then $\chi_{\{t \ge t_j\}} \p_\nu^k w_j(t,y) = 0$ for all $k$ and for all $j \geq 2$.
\end{lemma}
\begin{proof}	
We write $B_j = B_R(x_j)$, $j=1,\dots,J$.
	Since $d(y,S_j) \geq d(y,B_j)$, by finite speed of propagation, it will be sufficient to show
	\begin{align}\label{fsp2}
          		d(y,B_j) \ge t - t_j, \quad t \ge t_j,\ j \geq 2.
	\end{align}     
Let $z$ be the closest point to $y$ in $B_j$.  
By Lemma \ref{lem_SS}, there is $x \in B_1$ such that $d(z, x) \le \rho t_j$.
Thus 
$$
t - d(y, B_1) \le (1 - \rho)t_2 \le (1-\rho)t_j \le t_j - d(z, x).
$$
Hence
$$
t-t_j \le  d(y, B_1) - d(z, x)
\le d(y,x) - d(z,x) \le d(y,z) = d(y,B_j).
$$
\end{proof}

\begin{figure}
		\centering
        		\includegraphics[width=10cm]{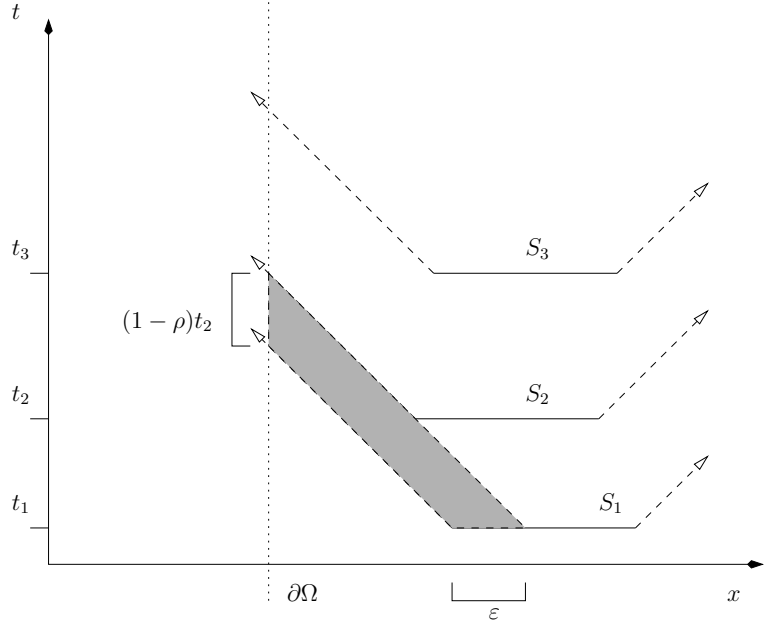}
	\caption{The gray area is affected only by the first source. Here, $t_1~=~0$.}
	\label{fspucp_fig}
\end{figure}

\begin{lemma}\label{ucp}
Let $x_j \in \Omega$, $j=1,2,\dots,J$ satisfy (\ref{SS}) 
and define $\rho \in [0,1)$ and $R > 0$ by (\ref{def_rho}).
We write
\begin{align}
\label{def_eps}
\vareps_0 = (1-\rho)t_2, \quad B_1 = B_R(x_1),
\end{align}
and let $\vareps \in (0,\vareps_0]$.
If $T > \max_{y \in \p \Omega} d(y,B_1) + \vareps$
then $f_1$ is uniquely determined by $\Lambda F$ in the interior of the set 
$$
\Omega_\vareps = 
\set{x \in \Omega;\ \text{there is $y \in \partial \Omega$ such that $d(x,y) \leq d(y,B_1) + \vareps$}}.
$$
\end{lemma}
\begin{proof}
By solving the exterior problem 
\begin{align*}
\begin{cases}
\partial^2_{t} u - c(x)^2 \Delta u = 0 & \textrm{ in } (0,T) \times \Rn \setminus \Omega, 
\\
u|_{x \in \p \Omega} = \Lambda F & \textrm{ in } (0,T) \times \p \Omega,
\\
u(0,\cdot) = \D_t u(0,\cdot) = 0 & \textrm{ in } \Rn \setminus \Omega,
\end{cases}
\end{align*} 
we recover $\p_\nu u|_{(0,T) \times \p \Omega}$.
We define $H_0 = (u, \p_\nu u)$ and
\begin{align*}
H(t,y) &= \begin{cases} H_0(t,y), &  t \in (0,T) \\
			-H_0(-t,y),&  t \in (-T,0),
			\end{cases}
\quad y \in \p \Omega.
\end{align*}
We set $h(y) = d(y,B_1) + \vareps$, $y \in \p \Omega$, and define $\Gamma(h)$
as in Theorem \ref{th_uniq_cont}.
Lemma \ref{fsp} implies that 
$H = (w_1, \p_\nu w_1)$ on 
$\Gamma(h) \cap (0,T) \times \p \Omega$,
and we have assumed that $\max_{y \in \p \Omega} h(y) < T$.
As $t_1 = 0$ and $w_1$ satisfies (\ref{PDEw}), $w_1$ is odd as a function of time. Therefore $H = (w_1, \p_\nu w_1)$ on $\Gamma(h)$,
and Theorem \ref{th_uniq_cont} implies that $f_1 = \partial_t w_1(0,\cdot)$ is uniquely determined by $H$ on the set $\Omega_\vareps^\inter$.
\end{proof}	 

\begin{proof}[Proof of Theorem \ref{uniqueVS}]
We use the notations from Lemma \ref{ucp}.
Recall that we have assumed $\vareps_0 > 2R$.
We take $\vareps = \diam{B_1} \le 2R$ and 
observe that $T$ satisfies the 
inequality in Lemma \ref{ucp} by 
the assumption $T > t_J + \diam{\Omega}$.
Lemma \ref{ucp} implies that $f_1$
is determined on $\Omega_\vareps$
and our choice of $\vareps$ implies that 
$B_R(x_1) \subset \Omega_\vareps$.
Thus $f_1$ is determined. 

We solve the wave equation (\ref{PDEu})
with $F$ replaced by $F_0 = \delta(t-t_1) f_1$.
Then we can determine $\Lambda F_1 = \Lambda F - \Lambda F_0$
where $F_1 = \sum_{j=2}^J\delta(t-t_j) f_j$.
We iterate the above steps to recover $f_2$, \dots, $f_J$.
\end{proof}

\begin{remark}
\label{rem_partial_data2}
Let us consider again the partial data case in 
Remark \ref{rem_partial_data}.
By that remark, we can recover the source times $t_j$, $j=1,2,\dots,J$.
Analogously to Lemma \ref{ucp} 
and Theorem \ref{uniqueVS}, it is possible to apply unique continuation to recover a part of $f_1$
and even the whole $F$ if a strong enough separation condition is satisfied. 
\end{remark}

\begin{lemma}\label{peel}
Suppose that $(\Omega, c^{-2} dx^2)$ is simple and 
define $\Omega_\vareps$ as in Lemma \ref{ucp}. Then 
$$
\Omega_\vareps = (B_R(0) \bs B_{R - \vareps}(0)) \cap \Omega.
$$
\end{lemma}
\begin{proof}	
	  Clearly $\Omega \cap B_{ R-\vareps}(0)^\inter \subset \Omega \setminus \Omega_\vareps$ even if simplicity is not assumed. Let $z \in \Omega \setminus \Omega_\vareps$ and choose $\xi \in S_0 \Omega$ and $s \ge 0$ such that $z = \exp_0(s \xi)$. It is sufficient to show that $z \in B_{ R-\vareps}(0)^\inter$.
We define $x$ and $y$ by (\ref{xy_radial}), i.e. $x = \exp_0(R\xi)$ and $y = \exp_0(\tau_0(\xi)\xi)$. 

	First, suppose that $s \ge R$. As $z$ is in between $x$ and $y$ on the geodesic $t \mapsto \exp_0(t\xi)$, and as all the geodesics are distance minimizing on a simple manifold, we have 
	\begin{align*}
		d(z,y) \le d(x,z) + d(z,y) = d(x,y) = d(y,B_1), 
	\end{align*}
	which contradicts $z \in \Omega \setminus \Omega_\vareps$, and therefore we have shown that $s < R$.

	Next, as $x$ is in between $z$ and $y$ on the geodesic $t \mapsto \exp_p(t\xi)$, we have
	\begin{align*}
		d(z,y) = d(z,x) + d(x,y) = d(z,x) + d(y,B_1).
	\end{align*}
	Moreover, $z \in \Omega \setminus \Omega_\vareps$ implies that 
	\begin{align*}
		d(y,B_1) < d(z,y) - \vareps.
	\end{align*}
	Hence $\vareps < d(z,x) = R - s$, and therefore $s < R - \vareps$.
Thus $z \in B_{ R-\vareps}(0)^\inter$.
\end{proof}	

\begin{proof}[Proof of Theorem \ref{uniqueRie}]
We choose $\epsilon = \min(\epsilon_0, R)$,
and observe that $T$ satisfies the 
inequality in Lemma \ref{ucp} by (\ref{T_cond}).
We recall the assumption that $S_j \subset \Omega$.
By Lemmas \ref{ucp} and \ref{peel}, 
$f_1$ is uniquely determined on the set $B_R(0) \bs B_{R - \vareps}(0)$.
By (\ref{R1}) the function $f_j$ is obtained from $f_1$ via the translation
$
\exp_{x_j} \circ \mathcal T_{x_j} \circ \exp_{0}^{-1}.
$
This translation maps $B_R(0) \bs B_{R - \vareps}(0)$
to $$A_j := B_R(x_j) \bs B_{R - \vareps}(x_j),$$
and therefore we can determine $f_j|_{A_j}$.  

We solve the wave equation (\ref{PDEu})
with $F$ replaced by $F_0(t,x) = \sum_{j=1}^J \delta(t-t_j) f_j|_{A_j}(x)$.
Then we can determine
$
\Lambda F - \Lambda F_0 = \Lambda F_1$,
where $F_1(t,x) = \sum_{j=1}^J \delta(t-t_j) \tilde f_j(x)$
and $\tilde f_j$ is the restriction of $f_j$ on $B_{R - \vareps}(x_j)$.
If $\epsilon = R$ then we have recovered $F$, 
otherwise we repeat the above construction starting from $\Lambda F_1$.
This iteration allows us to decrease the radius $R$ by $(1-\rho)t_2$ in each step (see Figure \ref{peeling_fig}), and therefore it will terminate in a finite number of steps. 
\end{proof}

\begin{figure}
\centering
  \includegraphics[width=10cm]{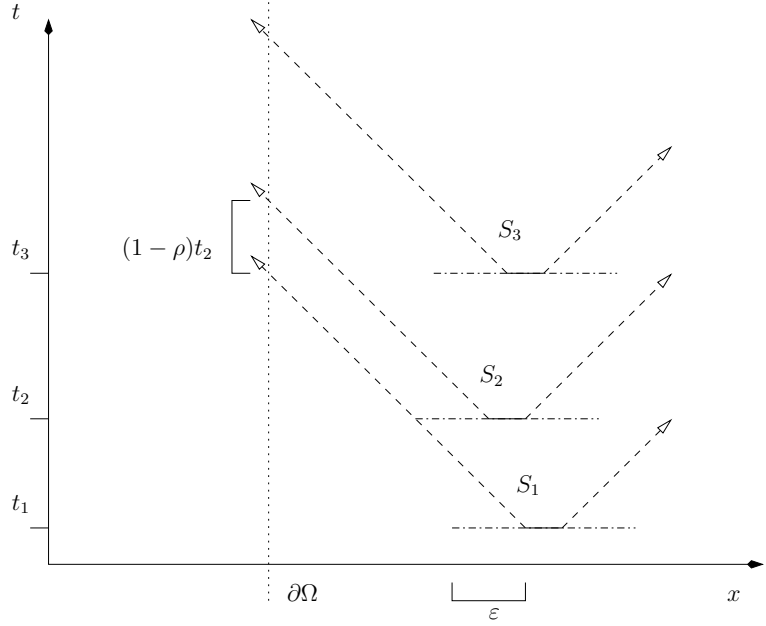}
\caption{At each step of the iteration, the radius where $f_1$ is unknown decreases by $(1-\rho)t_2$.}
\label{peeling_fig}
\end{figure}

\begin{proof}[Proof of Theorem \ref{uniqueEuc}]	
As before, $f_1$  is uniquely determined on the set  $B_R(0) \bs B_{R - \vareps}(0)$. We may assume without loss of generality that $\epsilon < R$.
Let us denote by $B^E_r(x)$ the Euclidean ball of radius $r$ centered at $x$.
As the geodesic ball $B_{R - \vareps}(0)$ is contained in the Euclidean ball $B^E_{c^+(R - \vareps)}(0)$,
we know $f_1$ outside $B^E_{c^+(R - \vareps)}(0)$.
The translation assumption (\ref{E1}) implies that 
$f_j$ is known outside $B^E_{c^+(R - \vareps)}(x_j)$.
This last ball is contained in the geodesic ball $B_{R^{(1)}}(x_j)$ where $R^{(1)} = \frac{c^+}{c^-}(R - \vareps)$.
As above we may remove the contribution of the known part of the functions $f_j$ from the data $\Lambda F$ and iterate the construction.

We terminate the iteration if $R^{(n)} \le \vareps$. Otherwise we set 
$R^{(n+1)} = \frac{c^+}{c^-}(R^{(n)} - \vareps)$
and reduce to the case $S_j \subset B_{R^{(n+1)}}(x_j)$. 
We have
\begin{align*}
R^{(n)} - R^{(n+1)} &= 
\frac{c^+}{c^-} 
\left(\epsilon - \left(1 - \frac{c^-}{c^+}\right) R^{(n)} \right).
\end{align*}
The assumption (\ref{E2}) implies that 
$$
\left(1 - \frac{c^-}{c^+}\right) R < (1-\rho)t_2 = \epsilon.
$$
Thus the sequence $R^{(n)}$ is decreasing and
\begin{align*}
R^{(n)} - R^{(n+1)} &\ge
\frac{c^+}{c^-} 
\left(\epsilon - \left(1 - \frac{c^-}{c^+}\right) R \right).
\end{align*}
Thus each step of the iteration decreases the radius by an amount that is bounded from below by a strictly positive constant, and therefore the iteration terminates in a finite number of steps.
\end{proof}

\section{Determining the translations from the supports $S_j$}
\label{sec_center}

Let us begin by considering the Euclidean translation condition \eqref{E1}.
Suppose that we know the sets $S_j$, $j=1,2,\dots,J$.
We define the center of mass
$$
\tilde x_j = \frac{1}{|\Sigma_j|} \int_{\Sigma_j} x dx.
$$
where $|\Sigma_j|$ is the Euclidean $n-1$ dimensional volume of $\Sigma_j$, $dx$ is the Euclidean surface measure on $\Sigma_j$ and $\Sigma_j$ is defined by (\ref{def_Sigma}).
By \eqref{E1} the function $f_j$ is obtained from $f_1$ via the translation
$\mathcal T_j^E(x) = x + x_j - x_1$.
Also the centers of mass are mapped via this translation, whence
$\tilde x_j - \tilde x_1 = x_j - x_1$.
Thus we can determine the translations $\mathcal T_j^E$ given the supports $S_j$ for all $j=1,2,\dots,J$.
When applying Theorem \ref{uniqueEuc} to recover the source $F$, we may assume that $x_j = \tilde x_j$ since this amounts to replacing $f$ with the translation $\tilde f(x) = f(x+\tilde x)$
where $\tilde x$ is the center of mass of $\supp{f}$. 

We turn now to the Riemannian translation condition \eqref{R1}, and consider only the case $\dim(S_j) = n$.
By \eqref{R1} the function $f_j$ is obtained from $f_1$ via the translation
$\mathcal T_j^R(x) = \exp_{x_j} \circ \mathcal T_{x_j} \circ \mathcal T_{x_1}^{-1} \circ \exp_{x_1}^{-1}.
$
We will give next a condition that guarantees that 
the translations $\mathcal T_j^R$ can be determined by using centers of mass analgously to the Euclidean case.

Let $\kappa$ and $K$ be a lower bound for the injectivity radius and an upper bound for the sectional curvature of the Riemannian manifold $(\Omega, c^{-2} dx^2)$, respectively, and define 
$$
r_\Omega = \min \{\kappa, \frac \pi {2 \sqrt{K}}\}.
$$
Suppose that $S \subset \R^n$ is measurable set that is contained in a geodesic ball
$B(p,r) \subset \Omega$ where $p \in \Omega$ and $r < r_\Omega$. Then the function $\varrho_S(x) = \max_{y \in S} d(x,y)$
has a unique minimizer $x_S$ (see \cite{afsari} Theorem 2.1).

Let us write $S = \supp{f}$ and denote by $|\xi|_g$ the norm of $\xi \in T_0\Omega$ with respect to the Riemannian metric $g = c^{-2} dx^2$.
We suppose that there is $R \in (0, r_\Omega)$ such that
\begin{align}\tag{R2}\label{R2}
\text{(i) }&\text{$|\xi|_g \le R$ for all $\xi \in S$,
and}
\\\notag
\text{(ii) }&\text{there is $\xi_0 \in S_0 \Omega$
such that $R\xi_0 \in S$ and $-R\xi_0 \in S$.}
\end{align}
The condition (R2) implies that there are two points on the boundary of $S$ that are symmetric with respect to the origin.

\begin{lemma}
Suppose that \eqref{R1} and \eqref{R2} hold.
Then the minimizer $x_{S_j}$ is $x_j$.
\end{lemma}
\begin{proof}
For any $x \in \Omega$, the parallel translation $\mathcal T_{x}$ is a linear isometry,
and if $\xi \in T_x \Omega$ satisfies $|\xi|_g \le \kappa$ and $\exp_{x}(\xi) \in \Omega$ then $d(\exp_{x}(\xi), x) = |\xi|_g$.
Let $j = 1, \dots, J$ and define 
$x^\pm = \exp_{x_j}(\pm R \mathcal T_{x_j} \xi_0)$.
Then for all $x \in \Omega$
$$
d(x^+, x) + d(x^-,x) \ge d(x^+, x^-) = 2 R,
$$
and $\varrho_{S_j}(x) \ge R$. On the other hand, $S_j \subset B(x_j, R)$.
Hence $x_j$ is a minimizer of $\varrho_{S_j}$.
\end{proof}

\section{Examples}\label{examples}

If $\dim(S_j) = n$ then for a short time after $t_j$, the corresponding outward wavefront 
does not intersect $T^* S_j$, and the inward wavefront is contained in $T^* S_j$.
We will assume in this section that $\dim(S_j) = n$, $j=1,\dots,J$.
The condition (\ref{ML1}) can be seen as consisting of two requirements: first, that no outward propagating wavefront intersects any later wavefront, and second, that no inward propagating wavefront intersects any later wavefront.
We show below that the first part of  (\ref{ML1}) is implied by \eqref{SS} under some further conditions.

\begin{example}If $S_j = B_r(x_j)$, $j=1,2,\dots,J$, for some $r > 0$, then \eqref{SS} implies the first part of \eqref{ML1}.\end{example}

\begin{proof}
	To see this, note that the outgoing wavefront of $B_r(x_j)$ at time $t$ is $\partial B_{r + t - t_j}(x_j)$. Choose any $k > j$ and $x \in \overline{B_r(x_k)}$, by \eqref{SS}, there is some $y \in B_r(x_j)$ so that $d(x,y) < \rho\abs{t_j - t_k}$, and further, $d(y,x_j) < r$ so that $d(y,x_j) < r + \rho(t_k - t_j)$ showing 
$$x \in B_{r + \rho(t_k - t_j)}(x_j) \subset B_{r + t_k - t_j}(x_j)$$ so that the wavefront has already completely passed $S_k$ at $t = t_k$.
\end{proof}

\begin{example}Suppose that the Riemannian manifold $(\Omega, c^{-2} dx^2)$
is simple. If $S_j$ are arbitrary, and \eqref{TS} is satisfied, then the first part of \eqref{ML1} is satisfied.\end{example}

\begin{proof}
	To demonstrate this claim, suppose that an outgoing ray from $x\in\partial S_j$ intersects $S_k$ at some point $y$ at time $t$. If $t < t_k$ then there is nothing to verify, so assume $t \geq t_k$ (if we can show intersections do not happen on the base manifold, then they do not happen in the cotangent bundle either). Then on one hand, $d(y,x) = t - t_j \geq t_k - t_j$, and on the other
$$d(y,x) \leq d(y,x_k) + d(x_k,x_j) + d(x_j,x) \leq d(x_k,x_j) + 2R \leq \rho(t_k - t_j) + 2R.$$ 
Then, by \eqref{TS}, 
$$2R < (1 - \rho)(t_{j+1} - t_j) \leq (1 - \rho)(t_k -t_j) $$
so that finally, 
$$d(y,x) \leq \rho(t_k - t_j) + 2R < t_k - t_j$$
which is a contradiction.
\end{proof}

\begin{figure}
\centering
  \includegraphics[width=5.5cm]{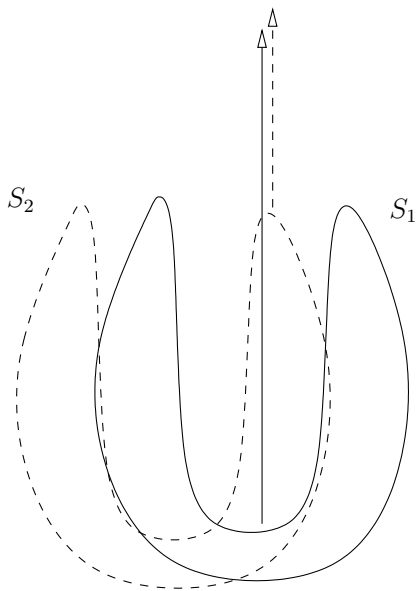}
\caption{A non-convex set.}
\label{horseshoe}
\end{figure}

Further, \eqref{TS} implies \eqref{SS}, so that if the second part of \eqref{ML1}, \eqref{ML2} and \eqref{TS} are assumed, then all the hypotheses for Theorems \ref{th_microlocal} and \ref{uniqueVS} are satisfied, yielding a complete reconstruction.

For the next example, let $r_c$ be the maximum $r$ such that $B_r(x)$ is convex for every $x \in \Omega$. This is known as the convexity radius of $\Omega$, and it is positive for any compact manifold (see \cite{berger}, Proposition $95$).

\begin{example}
\label{example_convex} 
If $S_j$ are convex, and $t_J - t_1 < r_c$, then \eqref{SS} implies the first part of \eqref{ML1}.\end{example}

To see that convexity is essential in Example \ref{example_convex}, consider Figure \ref{horseshoe}. Here, for a non-convex ``horseshoe'' shaped $S_1$, a ray leaving the ``bend'' of the shoe intersects the ``prong'' at a time later than $t_2$.
The proof of Example \ref{example_convex} is based on the following lemma.
\begin{lemma}\label{shortlemmaglobal}
	Let $C$ be a convex set in a Riemannian manifold $(M,g)$, let $A = \partial C$, let $y \in M \bs C$, and let $\sigma$ be a geodesic from $y$ to some point in $x \in A$ such that $\sigma$ is normal to $A$ at $x$ and such that $d(x,y) < r_c$. Then $\sigma$ minimizes the distance from $y$ to $C$.
\end{lemma}
\begin{proof}
	For contradiction, assume there is some point $z \in A$ so that $d(y,z) < d(y,x)$.
	
	Consider the totally geodesic hyperplane $\S$ tangent to $A$ at $x$. Because $C$ is strictly convex, it lies entirely on one side of $\S$; call this side $H_1$, and the other $H_2$ and note that both are convex. Let $B = B_{d(y,x)}(y)$; as a radial geodesic, $\sigma$ is normal to $\partial B$ at $x$, and thus $\S$ is tangent to $B$ as well. Because $d(y,x) < r_c$,  $B$ is convex and must also lie entirely on one side of $\S$. 
	
	For some $s_1 > d(y,x)$, $\sigma(s_1) \in H_1$, and for some $s_2 < d(y,x)$, $\sigma(s_2) \in H_2$. Thus we must have $y \in H_2$, otherwise $\sigma$ is a geodesic that exits and then re-enters $H_1$, violating convexity. Thus $B$ and $C$ lie on opposite sides of $\S$.
	
	Therefore, since $B_{d(y,z)}(y) \subset B$, $z$ cannot be in $A$, and we have a contradiction.
\end{proof}

\begin{proof}[Proof of Example \ref{example_convex}]
	Now, suppose that an outgoing ray from $x\in\partial S_j$ intersects $S_k$ at some point $y$ at time $t$ (as before, it is sufficient to show intersections do not happen in the base manifold). If $t < t_k$ then there is nothing to verify, so assume $t \geq t_k$. Then $d(y,x) = t - t_j < r_c$ so $d(y,x) = d(y,S_j)$ by Lemma \ref{shortlemmaglobal}. On the other hand, $t - t_j \geq t_k - t_j > \rho(t_k - t_j)$ which violates \eqref{SS}.    
\end{proof}

\section*{Acknowledgments}

The authors express their gratitude to the Institut Henri Poincar\'e and the organizers of the program on Inverse problems in 2015 for providing an excellent research environment while part of this work was in progress. LO was partially supported by the EPSRC grant EP/L026473/1 and Fondation Sciences Math\'ematiques de Paris. JT was supported in part by the members of the Geo-Mathematical Imaging Group at Purdue University.

\bibliographystyle{cj}

\bibliography{tittelfitz}

\end{document}